\documentclass[11pt]{amsart}

\usepackage[T1]{fontenc}
\usepackage[utf8]{inputenc}
\usepackage{amsmath,amssymb,amsfonts,amsthm}
\usepackage{mathtools}
\usepackage{enumitem}

\newtheorem{theorem}{Theorem}[section]
\newtheorem{lemma}[theorem]{Lemma}
\newtheorem{proposition}[theorem]{Proposition}
\newtheorem{corollary}[theorem]{Corollary}
\theoremstyle{definition}

\theoremstyle{remark}
\newtheorem{remark}[theorem]{Remark}

\newcommand{\N}{\mathbb N}

\newcommand{\sgn}{\operatorname{sign}}
\newcommand{\e}{\mathrm e}
\newcommand{\ii}{\mathrm i}
\newcommand{\Vlog}{V_{\log}}

\begin{document}

\title[Logarithmic estimates of Vilenkin Lebesgue constants]
{Two-sided logarithmic estimates of the Lebesgue constants with respect to Vilenkin systems}

\author{G. Tephnadze}
\address{G. Tephnadze, The University of Georgia, School of Science and Technology, 77a Merab Kostava St, Tbilisi, 0128, Georgia.}
\email{giorgitephnadze@gmail.com}

\date{}

\begin{abstract}
We introduce a new logarithmic digit variation for Vilenkin--Fourier partial sums and establish two-sided estimates for the corresponding Lebesgue constants with universal constants. These estimates hold for both bounded and unbounded Vilenkin systems.
As an application, we show that bounded logarithmic variation yields
uniform boundedness of the corresponding partial-sum operators on \(L_1\),
\(L_1\)-norm convergence for every integrable function, and almost
everywhere convergence.
\end{abstract}

\subjclass[2020]{42C10, 43A50}
\keywords{Vilenkin system, Dirichlet kernel, Lebesgue constant, two-sided estimate, logarithmic digit variation, partial sums, modulus of continuity, $L_1$-norm convergence, almost everywhere convergence, maximal operator, zero-dimensional group}

\maketitle

\section{Introduction}

In 1947 Vilenkin \cite{Vilenkin} introduced a large class of complete
orthonormal systems on zero-dimensional compact groups. These systems contain
the Walsh system as the special case $m_j=2$ for all $j$ and have become a
basic model in dyadic and martingale harmonic analysis; see, for example,
\cite{AVD,Golubov,Schipp,PTW}. Let $G_m$ be the Vilenkin group generated by
a sequence $m=(m_0,m_1,\ldots)$ of integers $m_j\ge2$. For
$f\in L_1(G_m)$, write
$$
\widehat f(k):=\int_{G_m}f\overline{\psi_k}\,d\mu,
\ \
S_nf:=\sum_{k=0}^{n-1}\widehat f(k)\psi_k=D_n\ast f,
\ \
D_n:=\sum_{k=0}^{n-1}\psi_k,
$$
and denote by
\[
L_n:=\|D_n\|_1
\]
the $n$th Lebesgue constant. The arithmetic structure of $n$ reflected in
$L_n$ is closely connected with boundedness, norm convergence, and
pointwise convergence of Vilenkin--Fourier partial sums.

For $1<p<\infty$,  \(L_p\)-norm convergence of Fourier series with respect to
Walsh and Vilenkin systems was developed by Schipp \cite{Schipp1976}, Simon
\cite{Simon1976}, and Young \cite{Young1976}. In this range the partial
sum operators are uniformly bounded in $L_p$, and the corresponding Fourier
series converge in norm. The endpoint $p=1$ is essentially different.
Watari \cite{Watari1958} (see also Gosselin \cite{Gosselin} and Young
\cite{Young1976}) obtained weak-type estimates for partial sum operators,
whereas uniform strong $L_1$ boundedness fails in general. In fact, Vilenkin
systems do not form Schauder bases in $L_1$; see \cite{AVD,Golubov}.

Approximation properties at this endpoint are naturally described by moduli
of continuity. 
For $f\in L_p(G_m)$, the integrated modulus of continuity at the Vilenkin
scale $1/M_r$ is
\begin{equation}\label{def:omega-p}
\omega_p\left(\frac1{M_r},f\right)
:=\sup_{h\in I_r}\|f(\cdot+h)-f(\cdot)\|_p.
\end{equation}
The standard approximation relations are (see Watari \cite{Watari}
and the discussion in \cite{BPT})
\begin{equation}\label{eq:omega-SM}
\frac12\,\omega_p\left(\frac1{M_r},f\right)
\le \|f-S_{M_r}f\|_p
\le \omega_p\left(\frac1{M_r},f\right).
\end{equation}
In particular,
\[
\omega_p\left(\frac1{M_r},f\right)\rightarrow0
\qquad (r\to\infty),\qquad 1\le p<\infty.
\]
Onneweer \cite{Onneweer} proved a sharp modulus-of-continuity criterion for
$L_1$-norm convergence of Walsh--Fourier series. For bounded Vilenkin
systems, Blahota, Persson and Tephnadze \cite{BPT} related the norm
behavior of subsequences of partial sums to digit variation and obtained
corresponding modulus-of-continuity estimates. Related sharp results in the
martingale $H_1$ setting were obtained by Tephnadze
\cite{Tephnadze2014}; see also the monograph \cite{PTW}.

The pointwise convergence problem has an even longer and more celebrated
history. Luzin \cite{Luzin} formulated the almost everywhere convergence
problem for Fourier series of $L_2$ functions. Kolmogorov
\cite{Kolmogorov} showed that no analogous theorem can hold for all
$L_1$ functions by constructing an integrable function whose trigonometric
Fourier series diverges almost everywhere. Carleson \cite{Carleson}
settled Luzin's problem affirmatively for $L_2$, and Hunt \cite{Hunt}
extended the result to $L_p$, $1<p<\infty$. Fefferman
\cite{Fefferman} subsequently gave a new proof based on estimates for the
maximal operator of partial sums. The Carleson--Hunt theorem has inspired
many further approaches and expositions, including those of Kahane
\cite{Kahane}, Lacey and Thiele \cite{LaceyThiele}, and Demeter
\cite{Demeter}.

Analogous questions for Walsh and Vilenkin systems have been studied by many
authors. Billard \cite{Billard} proved the Walsh analogue for $p=2$ and
Sj\"olin \cite{Sjolin} treated $1<p<\infty$, while Gosselin
\cite{Gosselin} obtained the corresponding result for bounded Vilenkin
systems. Martingale and tree-martingale methods developed by Schipp and
Weisz provide another important approach; see \cite{SchippWeisz}. A
modern martingale proof of the Vilenkin analogue of the Carleson--Hunt
theorem, together with a Kolmogorov-type divergence result at $p=1$, was
given by Persson, Schipp, Tephnadze and Weisz \cite{PSTW}. Near the
$L_1$ endpoint, Antonov \cite{Antonov} obtained a refined convergence
result for Walsh--Fourier series, and Oniani \cite{Oniani} established a
corresponding extrapolation theorem for bounded Vilenkin systems.

Lebesgue constants provide a natural bridge between the arithmetic of the
indices and norm and almost everywhere convergence. Earlier investigations
of generalized Walsh systems associated with products of cyclic groups, 
including the case of unbounded orders, were carried out
by Price \cite{Price}. 
For bounded Vilenkin systems, estimates of $L_n$ in
terms of digit variations were obtained by Lukyanenko \cite{Lukyanenko} and
Lukomskii \cite{Lukomskii}, and later sharpened in \cite{BPT}. 
In particular, if
$
n=\sum_{j=0}^{\infty}n_jM_j,$ 
and $
\delta_j:=\sgn(n_j),
$ 
$$
v(n):=\delta_0+\sum_{j=0}^{\infty}|\delta_{j+1}-\delta_j|,
$$
then the bounded-system estimates involve, besides $v(n)$, the additional
quantity
\[
v^*(n):=\sum_{j=0}^{\infty}|\ominus n_j-1|\,\delta_j.
\]

For subsequences of Walsh--Fourier partial sums, bounded digit variation is
a classical sufficient condition for almost everywhere convergence, although
Konyagin \cite{Konyagin} showed that it is not necessary. In the general
Vilenkin setting, Areshidze, Persson and Tephnadze \cite{APT} proved,
for bounded as well as unbounded systems, a weak-type $(1,1)$ estimate and
almost everywhere convergence of $\{S_{\alpha_k}f\}$ under the condition
\[
\sup_k\bigl(v(\alpha_k)+v^*(\alpha_k)\bigr)<\infty.
\]
Their proof is based on a pointwise decomposition of Dirichlet kernels and
the weak-type $(1,1)$ inequality for the martingale maximal operator
\begin{equation*}
\mathcal Mf(x):=
\sup_{j\in\N}S_{M_j}|f|(x)
=
\sup_{j\in\N}M_j\int_{I_j}|f(x-t)|\,d\mu(t).
\end{equation*}
In particular, there exists an absolute constant $C$ such that (see e.g. 
\cite{PTW,Shalukhina})
\begin{equation}\label{eq:mart-max-weak}
\sup_{\lambda>0}\lambda\,
\mu\{x\in G_m:\mathcal Mf(x)>\lambda\}
\le C\|f\|_1.
\end{equation}
This estimate is valid without any boundedness assumption
on the generating sequence $m$.

More recently, Pe\v{c}enkovi\'c \cite{Pecenkovic} derived an explicit
formula for the Lebesgue constants on general Vilenkin groups in terms of
the nonzero digits of the index and certain finite exponential sums.

The purpose of the present paper is different: we identify a simple
logarithmic digit variation which is uniformly equivalent to the Lebesgue
constants, with absolute constants independent of the generating sequence.
This provides a direct description of their order of magnitude and leads
to applications to norm convergence and almost everywhere convergence.

For $n_j\ne0$, put
$
d_j:=\min\{m_j-n_j-1,n_j+1\},
$
and define
\begin{equation*}
\Vlog(n):=v(n)+\sum_{\substack{j\ge0\\ n_j\ne0}}\log(1+d_j).
\end{equation*}

For every nonzero digit \(n_j\), we have
$
d_j\le m_j-n_j-1,
$ 
$
\log(1+d_j)\le d_j.$
Hence
$
V_{\log}(n)
\le
v(n)+v^*(n).
$
Thus boundedness of \(v(n)+v^*(n)\) implies boundedness of
\(V_{\log}(n)\), while the converse need not hold for unbounded
Vilenkin systems.

Our main result is the following.

\begin{theorem}\label{thm:main} 
Let \(G_m\) be an arbitrary Vilenkin group, where only \(m_j\ge2\) is assumed. Then, for every \(n\in\mathbb N\), 
\begin{equation}\label{eq:main} 
\frac{1}{384\pi}\,V_{\log}(n) \le L_n=\|D_n\|_1 \le \frac{7}{2}\,V_{\log}(n).
\end{equation} 
\end{theorem}

Theorem~\ref{thm:main} shows that $\Vlog(n)$ is the correct uniform
arithmetic scale for Lebesgue constants even when the local bases $m_j$ are
unbounded. It also leads directly to norm convergence. We prove a
quantitative estimate in terms of \eqref{def:omega-p} and show that, for
every increasing sequence $\{\alpha_k\}$, boundedness of
$\Vlog(\alpha_k)$ is equivalent to uniform boundedness of
$S_{\alpha_k}:L_1\to L_1$ and, when $\alpha_k\to\infty$, to
\[
\|S_{\alpha_k}f-f\|_1\longrightarrow0
\qquad\text{for every }f\in L_1(G_m).
\]
Thus the digit variation controlling the Lebesgue constants also determines
the $L_1$ behavior of subsequences.

The same logarithmic variation controls almost everywhere convergence.

\begin{theorem}\label{thm:ae}
Let $G_m$ be an arbitrary Vilenkin group, bounded or unbounded, and let
$\{\alpha_k:k\in\N\}$ be a strictly increasing sequence of positive
integers. Assume that
\begin{equation}\label{eq:Vlog-bounded}
K:=\sup_{k\in\N}\Vlog(\alpha_k)<\infty.
\end{equation}
Then there exists a constant $C_K>0$, depending only on $K$, such that for
every $f\in L^1(G_m)$,
\begin{equation}\label{eq:weak-subsequence}
\sup_{\lambda>0}\lambda\,
\mu\left\{x\in G_m:
\sup_{k\in\N}|S_{\alpha_k}f(x)|>\lambda\right\}
\le C_K\|f\|_1.
\end{equation}
Consequently,
\begin{equation}\label{eq:ae-subsequence}
S_{\alpha_k}f(x)\rightarrow f(x)
\qquad\text{for a.e. }x\in G_m \qquad \text{as } \qquad  k\to\infty.
\end{equation}

\end{theorem}

For unbounded Vilenkin systems, condition \eqref{eq:Vlog-bounded} may be
substantially weaker than the earlier condition involving $v+v^*$. For
example, for $\alpha_k=M_k$ one has
$\sup_k\Vlog(M_k)<\infty$, whereas $v^*(M_k)=m_k-2$ may be unbounded.
Thus Theorem~\ref{thm:ae} treats the martingale subsequence
$\{S_{M_k}\}$ uniformly even when the local bases are unbounded.

The proof of the two-sided estimate for the Lebesgue constants has two main steps. First, for a finite cyclic group of order $m$, we estimate the
normalized $L_1$-norm of a finite geometric sum by a logarithm of the shorter
cyclic length. Second, we derive an exact digit recursion for the increase of
$L_n$. The latter is particularly useful for the lower estimate because the
individual digit contributions are nonnegative and can therefore be summed
without cancellation.
Thus, the logarithmic digit variation introduced in this paper provides a unified criterion governing the size of Lebesgue constants, norm convergence in $L_1$, and almost everywhere convergence of subsequences for general Vilenkin systems.

The paper is organized as follows. Section~2 contains the notation and basic
identities for Vilenkin systems. Section~3 is devoted to finite cyclic
Lebesgue constants. In Section~4 we establish the exact digit recursion and
prove Theorem~\ref{thm:main}. Section~5 contains the $L_1$-norm convergence
results and their connection with the modulus of continuity. Finally, in
Section~6 we prove the maximal inequality and almost everywhere convergence
asserted in Theorem~\ref{thm:ae}.

\section{Preliminaries}

Let $\N_+:=\{1,2,\ldots\}$ and $\N:=\N_+\cup\{0\}$. Let
$m=(m_0,m_1,\ldots)$ be an arbitrary sequence of integers satisfying
$
m_j\ge2$ $(j\in\N).
$
No assumption is made on $\sup_jm_j$.
For each $j\in\N$, let
$$
Z_{m_j}:=\{0,1,\ldots,m_j-1\}
$$
be the additive cyclic group modulo $m_j$, and set
\[
G_m:=\prod_{j=0}^{\infty}Z_{m_j}.
\]
The group $G_m$ is equipped with the product topology and the normalized
Haar measure $\mu$, where each point of $Z_{m_j}$ has measure $1/m_j$.

For $x=(x_0,x_1,\ldots)\in G_m$, define
\[
I_0(x):=G_m,
\qquad
I_N(x):=\{y\in G_m:y_0=x_0,\ldots,y_{N-1}=x_{N-1}\},
\]
and write $I_N:=I_N(0)$. Then
$
\mu(I_N)=\frac1{M_N},
$
where
$$
M_0:=1,
\qquad
M_{N+1}:=m_NM_N.
$$
Every $n\in\N$ has a unique expansion
\begin{equation}\label{eq:digits}
n=\sum_{j=0}^{\infty}n_jM_j,
\qquad
0\le n_j<m_j,
\end{equation}
with only finitely many nonzero digits. 
For $n\ne0$ we put
$$
|n|:=\max\{j:n_j\ne0\}.
$$

The generalized Rademacher functions are
\[
r_j(x):=\exp\!\left(\frac{2\pi\ii x_j}{m_j}\right),
\]
and the Vilenkin functions are
\[
\psi_n(x):=\prod_{j=0}^{\infty}r_j(x)^{n_j},
\]
where $n$ is represented as in \eqref{eq:digits}. 

The Vilenkin system
$\{\psi_n:n\in\N\}$ is a complete orthonormal system in $L_2(G_m)$; see,
for example, \cite{AVD,Vilenkin}.
It is known that (see \cite{PTW})
\begin{equation}\label{eq:r}
\frac{1}{m}\sum_{u=0}^{m-1}
e^{\frac{2\pi \mathrm{i}ua}{m}}
=
\begin{cases}
1, & a\equiv 0 \pmod m,\\[2mm]
0, & a\not\equiv 0 \pmod m.
\end{cases}
\end{equation}

We also need the following identities for Dirichlet kernels (for details see \cite{PTW}):
For every $N\in\N$,
\begin{equation}\label{eq:DM}
D_{M_N}(x)=
\begin{cases}
M_N,&x\in I_N,\\
0,&x\notin I_N.
\end{cases}
\end{equation}
In particular,
\begin{equation}\label{eq:normDM}
\|D_{M_N}\|_1=1.
\end{equation}

If $0\le b<M_j$
and $1\le p<m_j$, then
\begin{equation}\label{eq:blockD}
D_{pM_j+b}
=
D_{M_j}\sum_{s=0}^{p-1}r_j^s+r_j^pD_b.
\end{equation}
Let $n=\sum_jn_jM_j.$ Then
\begin{equation}\label{eq:blockD1}
D_n = \psi_n\sum_{j=0}^{\infty} D_{M_j}\sum_{u=m_j-n_j}^{m_j-1}r_j^u.
\end{equation}

\section{Finite cyclic Lebesgue constants}

For $m\ge2$ and $0\le r\le m-1$, define
\begin{equation*}
\Lambda_m(r):=
\frac1m\sum_{a=0}^{m-1}
\left|
\sum_{u=1}^{r}\exp\!\left(\frac{2\pi\ii ua}{m}\right)
\right|,
\qquad
\Lambda_m(0):=0.
\end{equation*}
The following upper estimate will be used in the proof of the upper half of
Theorem~\ref{thm:main}.

\begin{lemma}\label{lem:Lambda-upper}
Let $m\ge2$ and $0\le r\le m-1$. Put
$
d:=\min\{r,m-r\}.
$
Then
\begin{equation*}
\Lambda_m(r)\le \frac72\log(1+d).
\end{equation*}
\end{lemma}

\begin{proof}
The assertion is trivial when $r=0$. Assume $1\le r\le m-1$. For
$1\le a\le m-1$, the finite geometric-sum formula gives
\[
\left|\sum_{u=1}^{r}\e^{2\pi\ii ua/m}\right|
=
\frac{|\sin(\pi ra/m)|}{\sin(\pi a/m)}.
\]
Hence
\begin{equation}\label{eq:Lambda-sine}
\Lambda_m(r)
=
\frac rm+
\frac1m\sum_{a=1}^{m-1}
\frac{|\sin(\pi ra/m)|}{\sin(\pi a/m)}.
\end{equation}

First suppose $1\le r\le m/2$, so $d=r$. By the symmetry
$a\leftrightarrow m-a$,
\[
\Lambda_m(d)
\le
\frac12+
\frac2m\sum_{a=1}^{\lfloor m/2\rfloor}
\frac{|\sin(\pi da/m)|}{\sin(\pi a/m)}.
\]
For $1\le a\le m/2$,
\[
\sin\frac{\pi a}{m}\ge \frac{2a}{m}
\]
and
\[
\left|\sin\frac{\pi da}{m}\right|
\le
\min\left\{1,\frac{\pi da}{m}\right\}.
\]
Consequently, with
$
N:=\left\lfloor\frac m2\right\rfloor,$ $c:=\frac{\pi d}{m},$
we obtain
\begin{equation}\label{eq:upper-min} \Lambda_m(d) \le \frac12+ \sum_{a=1}^{N} \min\left\{c,\frac1a\right\}. 
\end{equation} 
Since 
$ N=\left\lfloor\frac{m}{2}\right\rfloor \ge \frac{m}{3}, $
we have 
\[ cN = \frac{\pi d}{m} \left\lfloor\frac{m}{2}\right\rfloor \ge \frac{\pi d}{3} \ge \frac{\pi}{3} >1.\] 
Consider the decreasing function \[ g(x):= \min\left\{c,\frac1x\right\}, \qquad x\ge1. \] 

We distinguish two cases. 

If \(0<c\le1\), then \(1/c\ge1\), and since \(cN\ge1\), we also have $ \frac1c\le N.$ Hence 
\[ g(x)= \begin{cases} c, & 1\le x\le 1/c,\\[1mm] 1/x, & 1/c\le x\le N. \end{cases} \]
 Since \(g\) is decreasing, the standard integral comparison gives 
 \[ \sum_{a=1}^{N}g(a) \le g(1)+\int_1^N g(x)\,dx. \] 
 Consequently, 
\begin{align*} \sum_{a=1}^{N} \min\left\{c,\frac1a\right\} &\le c+ \int_1^{1/c} c\,dx + \int_{1/c}^{N}\frac{dx}{x}\\ &= c+ c\left(\frac1c-1\right) + \log N-\log\frac1c\\ &= c+(1-c)+\log(cN)= 1+\log(cN). 
\end{align*} 
If \(c\ge1\), then 
$ \frac1a\le1\le c $ $ (a\ge1), $
and therefore \[ \min\left\{c,\frac1a\right\} = \frac1a. \] Hence 
\[ \sum_{a=1}^{N} \min\left\{c,\frac1a\right\} = \sum_{a=1}^{N}\frac1a. \] 
By the standard estimate for the harmonic sum, 
\[ \sum_{a=1}^{N}\frac1a
\le 1+\int_1^N\frac{dx}{x}= 1+\log N. \] 
Since \(c\ge1\), we have 
$ \log N\le\log(cN), $
and consequently 
\[ \sum_{a=1}^{N} \min\left\{c,\frac1a\right\} \le 1+\log(cN). \] 
Thus, in both cases, 
\begin{equation}\label{eq:min-sum-est} 
\sum_{a=1}^{N} \min\left\{c,\frac1a\right\} \le 1+\log(cN). 
\end{equation} 
Finally, since 
$ N=\left\lfloor\frac{m}{2}\right\rfloor \le \frac{m}{2}, $
we obtain 
$ cN = \frac{\pi d}{m}N \le \frac{\pi d}{2}. $
Therefore, \[ \log(cN) \le \log\frac{\pi d}{2}. \] 
Combining this with \eqref{eq:min-sum-est}, we obtain 
\[ \sum_{a=1}^{N} \min\left\{c,\frac1a\right\} \le 1+\log\frac{\pi d}{2}. \] 
Finally, by \eqref{eq:upper-min}, 
\begin{equation}\label{eq:lambda} 
\Lambda_m(d) \le \frac32+\log\frac{\pi d}{2}.
\end{equation}

For $d\ge1$,
\begin{equation}\label{eq:log}
\frac32+\log\frac{\pi d}{2}
\le
\frac72\log(1+d).
\end{equation} 
Indeed, let \(d\ge 1\), define
\[
F(d):=
\frac{7}{2}\log(1+d)
-\frac{3}{2}
-\log\frac{\pi d}{2}.
\]
Then inequality~\eqref{eq:log} holds for \(d=1\), since
$F(1)=\frac72\log2-\frac32-\log\frac{\pi}{2}\approx 0.47>0.
$
On the other hand, the derivative of the difference
between the right- and left-hand sides equals
\begin{equation*}
F^{'}(d)=
\frac{7}{2(1+d)}-\frac1d
=
\frac{5d-2}{2d(1+d)}
\end{equation*} 
and
$F'(d)>0\quad(d\ge1).$

Hence \(F\) is increasing on \([1,\infty)\). Since \(F(1)>0\),
inequality~\eqref{eq:log} follows.

Now let \(m/2<r\le m-1\) and put \(d:=m-r\). Then, for
\(a=1,\ldots,m-1\),
\[
\left|
\sin\left(\frac{\pi r a}{m}\right)
\right|
=
\left|
\sin\left(
\pi a-\frac{\pi d a}{m}
\right)
\right|
=
\left|
\sin\left(\frac{\pi d a}{m}\right)
\right|.
\]
Hence the nonzero-frequency terms in \eqref{eq:Lambda-sine}
are unchanged if \(r\) is replaced by \(d\) and  subtracting $\Lambda_m(r)$ and $\Lambda_m(d)$ gives
\begin{eqnarray*}
\Lambda_m(r)-\Lambda_m(d)=\frac{r-d}{m},
\end{eqnarray*}
Hence, applying \eqref{eq:lambda}, we obtain
\begin{eqnarray}\label{eq:Lambda-symmetry}
\Lambda_m(r)
&=&
\Lambda_m(d)+\frac{r-d}{m}
=
\Lambda_m(d)+1-\frac{2d}{m}\\ \notag
&\le& \Lambda_m(d)+1
\le
\frac52+\log\frac{\pi d}{2}
\le
\frac72\log(1+d).
\end{eqnarray}
Indeed, if 
\[
h(d):=
\frac{7}{2}\log(1+d)
-\frac{5}{2}
-\log\frac{\pi d}{2},
\]
then
\[
h(2)=\frac72\log3-\frac52-\log\pi\approx0.2004>0,
\qquad
h'(d)=F'(d)>0\quad(d\ge2).
\]
Hence \(h\) is increasing on \([2,\infty)\). Since \(h(2)>0\),
the last inequality in \eqref{eq:Lambda-symmetry} follows.

If $d=1$, then $r=m-1$ and directly
\[
\Lambda_m(1)
=
\frac{1}{m}
\sum_{a=0}^{m-1}
\left|e^{2\pi i a/m}\right|=\frac{1}{m}\sum_{a=0}^{m-1}
1=1.
\]
Therefore, by \eqref{eq:Lambda-symmetry}, we obtain
\[
\Lambda_m(m-1)=\Lambda_m(1)+1-\frac2m=2-\frac2m<2<\frac72\log2.
\]
This completes the proof.
\end{proof}

For the lower estimate we use the closely related quantity
\begin{equation*}
\mathcal L_m(q):=
\frac1m\sum_{a=0}^{m-1}
\left|
\sum_{u=0}^{q-1}\exp\!\left(\frac{2\pi\ii ua}{m}\right)
\right|,
\qquad 1\le q\le m.
\end{equation*}
If $q<m$, then $\mathcal L_m(q)=\Lambda_m(q)$, while
$\mathcal L_m(m)=1$.

\begin{lemma}\label{lem:calL-lower}
Let $m\ge2$ and $2\le q\le m$. Put
$
d:=\min\{q,m-q\}.
$
Then
\begin{equation}\label{eq:calL-lower}
\mathcal L_m(q)-1
\ge
\frac{1}{96\pi}\log(1+d).
\end{equation}
\end{lemma}

\begin{proof}
If $d=0$, then $q=m$ and by \eqref{eq:r} we get that $\mathcal L_m(q)=1$, and both sides of
\eqref{eq:calL-lower} vanish. 


Suppose \(d=1\). Then
$d=\min\{q,m-q\}=1.$
There are two possibilities $q=1$ or $m-q=1.$
But \(q\ge2\), so \(q=1\) is impossible. Hence necessarily
$
m-q=1,
$
that is, $q=m-1.$

Now calculate \(\mathcal L_m(m-1)\). By definition,
\[
\mathcal L_m(m-1)
=
\frac1m
\sum_{a=0}^{m-1}
\left|
\sum_{u=0}^{m-2}e^{2\pi iua/m}
\right|.
\]

For \(a=0\),
\[
\sum_{u=0}^{m-2}e^{2\pi iu\cdot0/m}
=
\sum_{u=0}^{m-2}1
=
m-1.
\]
For \(a=1,\ldots,m-1\), we have \(a\not\equiv0\pmod m\). Hence, by
\eqref{eq:r}, we get that
$\sum_{u=0}^{m-1}e^{2\pi iua/m}=0.$
Therefore,
\[
\sum_{u=0}^{m-2}e^{2\pi iua/m}
=
-e^{2\pi i(m-1)a/m},
\]
and hence
$
\left|
\sum_{u=0}^{m-2}e^{2\pi iua/m}
\right|
=1.
$
There are \(m-1\) such nonzero values of \(a\). Consequently,
\[
\begin{aligned}
\mathcal L_m(m-1)=
\frac1m\left((m-1)+(m-1)\right)=
\frac{2(m-1)}{m}=
2-\frac2m.
\end{aligned}
\]

Thus,
\[\mathcal L_m(m-1)-1=1-\frac2m.\]

Since \(q=m-1\ge2\), we have
$
m\ge3.
$
Therefore,
$
1-\frac2m
\ge
1-\frac23
=
\frac13.
$
Finally,
\[\frac13>\frac{\log2}{96\pi},  \ \
\mathcal L_m(m-1)-1=1-\frac2m\ge\frac13>\frac{\log2}{96\pi}.\]

Because, by definition,
$
d=\min\{q,m-q\},
$
we always have
$
d\le \frac{m}{2}.$
Hence it remains to consider $2\le d\le m/2$.

We first reduce the proof to the case $q=d$. If $q>m/2$, then for every
$a\ne0$ the complete sum of all $m$-th roots of unity vanishes, so
\[
\left|\sum_{u=0}^{q-1}\e^{2\pi\ii ua/m}\right|
=
\left|\sum_{u=0}^{m-q-1}\e^{2\pi\ii ua/m}\right|.
\]

Indeed, for \(a\neq 0 \pmod m\), we have
$
\sum_{u=0}^{m-1} e^{2\pi iua/m}=0.
$
Hence
\[
\sum_{u=0}^{q-1}e^{2\pi iua/m}
=
-\sum_{u=q}^{m-1}e^{2\pi iua/m}.
\]
Putting \(u=q+v\), we obtain
\[
\sum_{u=q}^{m-1}e^{2\pi iua/m}
=
e^{2\pi iqa/m}
\sum_{v=0}^{m-q-1}e^{2\pi iva/m}.
\]
Since
$
\left|e^{2\pi iqa/m}\right|=1,
$
it follows that
\[
\left|
\sum_{u=0}^{q-1}e^{2\pi iua/m}
\right|
=
\left|
\sum_{v=0}^{m-q-1}e^{2\pi iva/m}
\right|.
\]
Thus, if \(q>m/2\), then \(d=m-q\) and
\[
\left|
\sum_{u=0}^{q-1}e^{2\pi iua/m}
\right|
=
\left|
\sum_{u=0}^{d-1}e^{2\pi iua/m}
\right|.
\]

Thus, all the terms with \(a\neq0\) in \(\mathcal L_m(q)\) and \(\mathcal L_m(d)\) are identical.
The only difference comes from \(a=0\). At \(a=0\),
\[
\left|
\sum_{u=0}^{q-1}1
\right|
=
q,
\qquad
\left|
\sum_{u=0}^{d-1}1
\right|
=
d.
\]
Therefore
\[
\mathcal L_m(q)-\mathcal L_m(d)
=
\frac{q-d}{m}.
\]

Now \(q=m-d\), hence
$
q-d=m-2d,
$
and therefore
\[
\mathcal L_m(q)-\mathcal L_m(d)
=
1-\frac{2d}{m}.
\]

Thus,
\[
\mathcal L_m(q)-1
=
\mathcal L_m(d)-1
+
1-\frac{2d}{m}.
\]

Finally, since \(d<m/2\), we have
$
1-\frac{2d}{m}\ge 0,
$
and consequently
\begin{equation}\label{eq:reduce-d}
\mathcal L_m(q)-1
=
\mathcal L_m(d)-1+1-\frac{2d}{m}
\ge
\mathcal L_m(d)-1.
\end{equation}
Thus, the case \(q>m/2\) can be reduced to the case
$
d=m-q<\frac{m}{2}.
$
Hence, after the above reduction, it is enough to consider
$
q=d,
$ $ 
2\le d\le \frac{m}{2}.
$

Thus we estimate $\mathcal L_m(d)-1$.
Put
\[
P_d(a):=\sum_{u=0}^{d-1}\e^{2\pi\ii ua/m}.
\]
Since
$
\frac1m\sum_{a=0}^{m-1}P_d(a)=1,
$
we have the identity
\begin{equation}\label{eq:Lminus1}
\mathcal L_m(d)-1
=
\frac1m\sum_{a=0}^{m-1}
\bigl(|P_d(a)|-\Re P_d(a)\bigr).
\end{equation}
Each summand on the right is nonnegative. Since
$
P_d(m-a)=\overline{P_d(a)},
$
we have
\[
|P_d(m-a)|-\Re P_d(m-a)
=
|P_d(a)|-\Re P_d(a).
\]
For \(1\le a\le m-1\), let
$x_a:=\frac{\pi a}{m}.$
Then
\[
P_d(a)=\sum_{u=0}^{d-1}e^{2\ii u x_a}
=
\frac{1-e^{2\ii d x_a}}{1-e^{2\ii x_a}}
=
\frac{-2\ii e^{\ii d x_a}\sin(dx_a)}
{-2\ii e^{\ii x_a}\sin x_a}
=
\e^{\ii(d-1)x_a}\frac{\sin(dx_a)}{\sin x_a},
\]
and therefore
\begin{equation}\label{eq:Pdiff}
|P_d(a)|-\Re P_d(a)
=
\frac{|\sin(dx_a)|}{\sin x_a}
\left[
1-\sgn(\sin(dx_a))\cos((d-1)x_a)
\right].
\end{equation}

We first establish a uniform positive lower bound. Let
\[
c_0:=1-\frac{\sqrt3}{2}
\qquad\text{and}\qquad
c_*:=\frac{\sqrt3\,c_0}{4\pi}.
\]
We claim that
\begin{equation}\label{eq:uniform-positive}
\mathcal L_m(d)-1\ge c_*.
\end{equation}
Set $h:=m/d\ge2$.

Assume first $h\ge3$. Consider the integers $a$ satisfying
\begin{equation}\label{eq:a-interval-uniform}
\frac{m}{3d}\le a\le\frac{2m}{3d}.
\end{equation}
The interval in \eqref{eq:a-interval-uniform} has length
$
\frac{2m}{3d}-\frac{m}{3d}
=
\frac{m}{3d}
=
\frac{h}{3}.
$
Since \(h\ge3\), we have \(h/3\ge1\). Any interval of length
\(L\ge1\) contains at least \(\lfloor L\rfloor\) integers, where \(\lfloor L\rfloor\) denotes the greatest integer not exceeding \(L\). Therefore,
the interval in \eqref{eq:a-interval-uniform} contains at least
$
\left\lfloor\frac{h}{3}\right\rfloor
\ge
\frac{h}{6},
$
integers, where in the last inequality we used
\(\lfloor x\rfloor\ge x/2\) for \(x\ge1\).

For such $a$,
\[
\frac{\pi}{3d}
\le
x_a
\le
\frac{2\pi}{3d}, \ \ \ \text{
and } \ \ \
\frac{\pi}{3}\le dx_a\le\frac{2\pi}{3},
\]
so $\sin(dx_a)\ge\sqrt3/2$. 
Moreover,

\[
\frac{\pi(d-1)}{3d}
\le
(d-1)x_a
\le
\frac{2\pi(d-1)}{3d}.
\]

Now, because \(d\ge2\),
$
\frac{d-1}{d}\ge \frac12,
$
and hence
$
\frac{\pi(d-1)}{3d}
\ge
\frac{\pi}{6}.
$
Also,
$
\frac{d-1}{d}<1,
$
so
$$
\frac{2\pi(d-1)}{3d}
<
\frac{2\pi}{3}.
$$
It follows that
$
\frac{\pi}{6}\le(d-1)x_a<\frac{2\pi}{3},
$
and hence
$
1-\cos((d-1)x_a)\ge c_0.
$
Finally,
$
\sin x_a\le x_a=\frac{\pi a}{m}\le\frac{2\pi}{3d}.
$
Thus \eqref{eq:Pdiff} yields
\begin{equation*}
|P_d(a)|-\Re P_d(a)
\ge
\frac{3\sqrt3\,c_0}{4\pi}\,d.
\end{equation*}

The terms corresponding to \(m-a\) have the same value, and the two sets of indices are disjoint. Hence, by \eqref{eq:Lminus1},

\[
\mathcal L_m(d)-1
\ge
\frac1m\cdot2\cdot\frac h6\cdot
\frac{3\sqrt3\,c_0}{4\pi}d
=
\frac{\sqrt3\,c_0}{4\pi}=c_*.
\]

Suppose now $2\le h<3$. Using only $a=1$ and $a=m-1$, we have
\[
\frac{\pi}{3}<\frac{\pi d}{m}\le\frac{\pi}{2},
\qquad
\frac{\pi}{6}<\frac{(d-1)\pi}{m}<\frac{\pi}{2}.
\]
Therefore
\[
\sin\frac{\pi d}{m}\ge\frac{\sqrt3}{2},
\qquad
1-\cos\frac{(d-1)\pi}{m}\ge c_0,
\qquad
\sin\frac{\pi}{m}\le\frac{\pi}{m}.
\]
The two terms $a=1,m-1$ in \eqref{eq:Lminus1} then give
\[
\mathcal L_m(d)-1
\ge
\frac{\sqrt3\,c_0}{\pi}>c_*.
\]
Thus \eqref{eq:uniform-positive} is proved.

It remains to derive the logarithmic lower bound. First suppose that \(2\le d\le23\). 
Since 
$ 96\pi c_* = 24\sqrt3\left(1-\frac{\sqrt3}{2}\right) > \log 24, $
by \eqref{eq:uniform-positive}
we obtain 
\[ \mathcal L_m(d)-1 \ge \frac{1}{96\pi}\log(1+d). \]
Suppose now that \(d\ge24\), and put 
$R:=\left\lfloor\frac{d}{12}\right\rfloor.$
Then \(R\ge2\). For \(0\le r\le R-1\), consider the set 
\[ A_r:= \left\{ a\in\mathbb N: \left(r+\frac14\right)h \le a\le \left(r+\frac34\right)h \right\}, \qquad h=\frac{m}{d}. \] 
Since in this part of the proof you have already reduced to
$
2\le d\le \frac m2,
$
we  get that the interval defining \(A_r\) has length \(h/2\ge1\), and hence, it contains at least 
$\left\lfloor\frac h2\right\rfloor\ge\frac h4 $
integers. For \(a\in A_r\), write 
\begin{equation}\label{eq:d-0}
\frac{a}{h}=r+s, \qquad \frac14\le s\le\frac34. 
\end{equation}
Since
$
a=h(r+s)=\frac{m}{d}(r+s),
$
we get
\begin{equation}\label{eq:d-1}
x_a=\frac{\pi a}{m}
=\frac{\pi(r+s)}{d}.
\end{equation}
Then $dx_a=\pi(r+s),$ and hence 
$ |\sin(dx_a)| = |\sin(\pi s)| \ge \frac{\sqrt2}{2}. $
Moreover,  \(\sin(\pi s)>0\),
$
\sin(dx_a)=sin(\pi(r+s))=(-1)^r\sin(\pi s),$
and therefore
$
\sgn(\sin(dx_a))=(-1)^r.
$
Moreover, 
$\sgn(\sin(dx_a))=(-1)^r. $
Using \eqref{eq:d-1} we get
\[ (d-1)x_a = \pi r+ \pi\left( s-\frac{r+s}{d} \right), \]
\[ \cos((d-1)x_a) = (-1)^r\cos\left( \pi\left( s-\frac{r+s}{d} \right) \right),\] 
and
\[ \sgn(\sin(dx_a))\cos((d-1)x_a) = \cos\left( \pi\left( s-\frac{r+s}{d} \right) \right). \] 
On the other hand, by using second estimate in \eqref{eq:d-0}, because \(r\le R-1\le d/12-1\), we get that
$$ s-\frac{r+s}{d}= s\left(1-\frac1d\right)-\frac rd\ge \frac14\left(1-\frac1d\right) -\left(\frac1{12}-\frac1d\right)= \frac16+\frac{3}{4d} > \frac16, $$ 
while clearly 
$ s-\frac{r+s}{d}\le s\le\frac34. $
Therefore, 
\[ 1-\sgn(\sin(dx_a))\cos((d-1)x_a) \ge 1-\cos\frac{\pi}{6} = c_0. \] 
Furthermore, 
\[ \sin x_a\le x_a = \frac{\pi(r+s)}{d} \le \frac{\pi(r+1)}{d}. \] 

Thus, by \eqref{eq:Pdiff}, \[ |P_d(a)|-\Re P_d(a) \ge \frac{\sqrt2\,c_0}{2\pi}\frac{d}{r+1}. \] 
By the above symmetry, the terms corresponding to \(m-a\) have the same
value. The intervals defining the sets \(A_r\) are mutually disjoint.
Moreover, if \(a\in A_r\), then
\[
\frac{a}{h}
\le r+\frac34
\le R-\frac14
<\frac{d}{12}+\frac34
<\frac d2,
\qquad d\ge24.
\]
Hence \(a<m/2\). Therefore, all sets \(A_r\) lie below \(m/2\),
whereas their reflected sets \(m-A_r\) lie above \(m/2\).
Consequently, the sets \(A_r\) and \(m-A_r\),
\(0\le r\le R-1\), are pairwise disjoint. Therefore, using
\eqref{eq:Lminus1},
\[ \begin{aligned} \mathcal L_m(d)-1 \ge \frac1m \sum_{r=0}^{R-1} 2\,\frac h4\, \frac{\sqrt2\,c_0}{2\pi} \frac{d}{r+1} = \frac{\sqrt2\,c_0}{4\pi} \sum_{r=0}^{R-1}\frac1{r+1} \ge \frac{\sqrt2\,c_0}{4\pi}\log(R+1).
\end{aligned}\]

Since 
$R=\lfloor d/12\rfloor\ge2,$ $d+1\le12(R+1)\le(R+1)^4,$
and consequently 
$ \log(R+1)\ge\frac14\log(d+1).$
It follows that 
$ \mathcal L_m(d)-1 \ge \frac{\sqrt2\,c_0}{16\pi}\log(d+1).$
Finally, $6\sqrt2\,c_0 = 6\sqrt2\left(1-\frac{\sqrt3}{2}\right)>1,$ and therefore \[ \mathcal L_m(d)-1 \ge \frac{1}{96\pi}\log(1+d). \] Together with \eqref{eq:reduce-d}, this proves \eqref{eq:calL-lower}.
The proof is complete.
\end{proof}

\section{Digit recursion and proof of the main theorem}

We begin with an exact recursion for the Lebesgue constants. For
$n=\sum_j n_jM_j$, set
\begin{equation*}
N_j:=\sum_{\ell=0}^{j}n_\ell M_\ell,
\qquad
N_{-1}:=0, \ \ \ \text{
and } \ \ \ 
t_j:=\frac{N_{j-1}}{M_j}.
\end{equation*}
Clearly $0\le t_j<1$.

For $m\ge2$, $1\le p<m$, and $0\le t<1$, define
\begin{equation*}
F_m(p,t):=
\frac1m\sum_{a=0}^{m-1}
\left|
\sum_{s=0}^{p-1}\e^{2\pi\ii sa/m}
+t\e^{2\pi\ii pa/m}
\right|-t.
\end{equation*}

\begin{proposition}\label{prop:recursion}
For every $n\in\N$,
\begin{equation*}
L_n=\sum_{\substack{j\ge0\\ n_j\ne0}}
F_{m_j}(n_j,t_j).
\end{equation*}
Moreover, for $1\le p<m$ and $0\le t<1$,
\begin{equation}\label{eq:F-first-lower}
F_m(p,t)\ge1-t
\end{equation}
and
\begin{equation}\label{eq:F-second-lower}
F_m(p,t)\ge \mathcal L_m(p+1)-1.
\end{equation}
\end{proposition}

\begin{proof}


For \(p=n_j\ne0\) we have
$
N_j=pM_j+N_{j-1},
$
$
t_j=\frac{N_{j-1}}{M_j}.
$
By using \eqref{eq:blockD} we get that
\[
D_{N_j}
=
D_{M_j}\sum_{s=0}^{p-1}r_j^s
+
r_j^pD_{N_{j-1}}.
\]

Now, we distinguish the regions \(G_m\setminus I_j\) and \(I_j\).

If \(x\notin I_j\), then
$
D_{M_j}(x)=0,
$
and hence
$
D_{N_j}(x)
=
r_j^p(x)D_{N_{j-1}}(x).
$
Since
$
|r_j^p(x)|=1,
$
we obtain
$
|D_{N_j}(x)|
=
|D_{N_{j-1}}(x)|.
$
Therefore,
\[
L_{N_j}-L_{N_{j-1}}
=
\int_{I_j}
\left(
|D_{N_j}(x)|
-
|D_{N_{j-1}}(x)|
\right)
\,d\mu(x).
\]

Now let \(x\in I_j\). Since \(N_{j-1}<M_j\), we have
$\psi_k(x)=1,$ $0\le k<N_{j-1},$
and consequently
$
D_{N_{j-1}}(x)
=
N_{j-1}.
$
Moreover,
$
D_{M_j}(x)=M_j.
$
Thus
\[
D_{N_j}(x)
=
M_j\sum_{s=0}^{p-1}r_j^s(x)
+
N_{j-1}r_j^p(x).
\]
Since
$
N_{j-1}=M_jt_j,
$
we get
\[
D_{N_j}(x)
=
M_j
\left(
\sum_{s=0}^{p-1}r_j^s(x)
+
t_jr_j^p(x)
\right).
\]

Next, partition \(I_j\) according to the value of the \(j\)-th coordinate:
\[
I_j
=
\bigcup_{a=0}^{m_j-1}
I_{j+1}(ae_j).
\]
For
$x\in I_{j+1}(ae_j),$
we have \(x_j=a\), and therefore
$
r_j(x)
=
e^{2\pi ia/m_j}.
$
Hence
\[
D_{N_j}(x)
=
M_j
\left(
\sum_{s=0}^{p-1}
e^{2\pi isa/m_j}
+
t_j e^{2\pi ipa/m_j}
\right),
\]
and consequently
\[
|D_{N_j}(x)|
=
M_j
\left|
\sum_{s=0}^{p-1}
e^{2\pi isa/m_j}
+
t_j e^{2\pi ipa/m_j}
\right|.
\]
On the other hand,
$
|D_{N_{j-1}}(x)|
=
N_{j-1}
=
M_jt_j.
$
Since
$
\mu\bigl(I_{j+1}(ae_j)\bigr)
=
\frac{1}{M_{j+1}}
=
\frac{1}{m_jM_j},
$
we obtain
\[
\begin{aligned}
L_{N_j}-L_{N_{j-1}}
&=
\sum_{a=0}^{m_j-1}
\frac{1}{m_jM_j}
\left[
M_j
\left|
\sum_{s=0}^{p-1}
e^{2\pi isa/m_j}
+
t_j e^{2\pi ipa/m_j}
\right|
-
M_jt_j
\right]
\\
&=
\frac1{m_j}
\sum_{a=0}^{m_j-1}
\left|
\sum_{s=0}^{p-1}
e^{2\pi isa/m_j}
+
t_j e^{2\pi ipa/m_j}
\right|
-t_j.
\end{aligned}
\]
By the definition of \(F_m(p,t)\), it follows that
$L_{N_j}-L_{N_{j-1}}=F_{m_j}(p,t_j).$

If \(J=|n|\), then \(n_j=0\) for all \(j>J\), and
$N_J=n.$
If \(n_j=0\), then
$
N_j=N_{j-1},
$
and consequently
$
L_{N_j}-L_{N_{j-1}}=0.
$
Therefore, summing from \(j=0\) to \(J\), we obtain
\[
\sum_{\substack{0\le j\le J\\ n_j\ne0}}
F_{m_j}(n_j,t_j)=\sum_{j=0}^{J}
\left(L_{N_j}-L_{N_{j-1}}\right)=L_{N_J}-L_{N_{-1}}=L_{N_J}=L_n.
\]
Hence
$	L_n=\sum_{\substack{j\ge0\\ n_j\ne0}}F_{m_j}(n_j,t_j).$

To prove \eqref{eq:F-first-lower}, set
$
Z_a:=
\sum_{s=0}^{p-1}e^{2\pi i s a/m}
+t e^{2\pi i p a/m},$ $ a=0,\ldots,m-1.$
Therefore
$$ \frac1m\sum_{a=0}^{m-1} Z_a = \frac1m\sum_{a=0}^{m-1} \left( \sum_{s=0}^{p-1}e^{2\pi i s a/m} +t\,e^{2\pi i p a/m} \right). $$

Interchanging the sums gives
$$ \frac1m\sum_{a=0}^{m-1} Z_a = \sum_{s=0}^{p-1} \frac1m\sum_{a=0}^{m-1}e^{2\pi i s a/m} + t\frac1m\sum_{a=0}^{m-1}e^{2\pi i p a/m}. $$

Since \(1\le p<m\), all indices $ s=1,\ldots,p-1 $
and also \(p\) lie between \(1\) and \(m-1\). According to \eqref{eq:r}, all corresponding averages are zero, the only nonzero term is the one with \(s=0\):
$$ \frac1m\sum_{a=0}^{m-1}e^{2\pi i\cdot 0\cdot a/m} = \frac1m\sum_{a=0}^{m-1}1 = 1. $$
Therefore,
$\frac1m\sum_{a=0}^{m-1} Z_a=1. $
By the triangle inequality,
\[
1
=
\left|
\frac1m\sum_{a=0}^{m-1} Z_a
\right|
\le
\frac1m\sum_{a=0}^{m-1}|Z_a|.
\]
Hence,
\[
F_m(p,t)
=
\frac1m\sum_{a=0}^{m-1}|Z_a|-t
\ge
1-t,
\]
which proves \eqref{eq:F-first-lower}. 

Now, we can write
$$ \sum_{s=0}^{p-1}e^{2\pi isa/m} +t e^{2\pi ipa/m} = \sum_{s=0}^{p}e^{2\pi isa/m} - (1-t)e^{2\pi ipa/m}. $$

By the reverse triangle inequality, $ |A-B|\ge |A|-|B|, $
so
$$ \left| \sum_{s=0}^{p-1}e^{2\pi isa/m} +t e^{2\pi ipa/m} \right| \ge \left| \sum_{s=0}^{p}e^{2\pi isa/m} \right| -(1-t). $$

Averaging over \(a=0,\ldots,m-1\), we obtain
$$ \frac1m\sum_{a=0}^{m-1} \left| \sum_{s=0}^{p-1}e^{2\pi isa/m} +t e^{2\pi ipa/m} \right| \ge \frac1m\sum_{a=0}^{m-1} \left| \sum_{s=0}^{p}e^{2\pi isa/m} \right| -(1-t). $$

By the definition of \(\mathcal L_m\),
$$ \frac1m\sum_{a=0}^{m-1} \left| \sum_{s=0}^{p}e^{2\pi isa/m} \right| = \mathcal L_m(p+1). $$

Hence
$$ \frac1m\sum_{a=0}^{m-1} \left| \sum_{s=0}^{p-1}e^{2\pi isa/m} +t e^{2\pi ipa/m} \right| \ge \mathcal L_m(p+1)-(1-t). $$

Subtracting \(t\) from both sides, we obtain
$$
F_m(p,t)
\ge
\mathcal L_m(p+1)-(1-t)-t.
$$
Since
$
(1-t)+t=1,
$
it follows that
$$
F_m(p,t)\ge \mathcal L_m(p+1)-1,
$$
which proves \eqref{eq:F-second-lower} and completes the proof.
\end{proof}


\begin{proposition}\label{prop:F-log} 
Let \(1\le p<m\), \(0\le t<1\), and put $d:=\min\{m-p-1,p+1\}.$ Then 
\begin{equation}\label{eq:F-log} 
F_m(p,t) \ge \frac{1}{192\pi}\log(1+d). 
\end{equation} 
If, in addition, \(t\le1/2\), then 
\begin{equation}\label{eq:F-block} 
F_m(p,t) \ge \frac{1}{192\pi} \bigl(1+\log(1+d)\bigr). 
\end{equation} 
\end{proposition} 
\begin{proof} By \eqref{eq:F-second-lower}, we have 
\[ F_m(p,t) \ge \mathcal L_m(p+1)-1. \] 
Applying Lemma~\ref{lem:calL-lower} with $ q=p+1, $ we obtain 
\[ \mathcal L_m(p+1)-1 \ge \frac{1}{96\pi} \log \left( 1+ \min\{p+1,m-p-1\} \right). \] 
Since $d=\min\{m-p-1,p+1\},$
it follows that 
\begin{equation}\label{eq:F}
F_m(p,t) \ge \frac{1}{96\pi}\log(1+d), 
\end{equation} 
so,
\[ F_m(p,t) \ge \frac{1}{192\pi}\log(1+d), \] 
which proves \eqref{eq:F-log}. 

Suppose now, in addition, that \(t\le1/2\). By \eqref{eq:F-first-lower}, \[ F_m(p,t) \ge 1-t \ge \frac12. \] 
If we combine this estimate with \eqref{eq:F} we get 
\[ F_m(p,t) \ge \max \left\{ \frac12, \frac{1}{96\pi}\log(1+d) \right\}. \] 
Using the elementary inequality \[ \max\{A,B\}\ge\frac{A+B}{2}, \qquad A,B\ge0, \] 
we obtain 
\[ \begin{aligned} F_m(p,t) \ge \frac12 \left( \frac12+ \frac{1}{96\pi}\log(1+d) \right)= \frac14+ \frac{1}{192\pi}\log(1+d). \end{aligned} \] 
Since $ \frac14 \ge \frac{1}{192\pi}, $ we conclude that 
\[ \frac14+ \frac{1}{192\pi}\log(1+d) \ge \frac{1}{192\pi} \bigl(1+\log(1+d)\bigr). \] 
Hence, 
\[ F_m(p,t) \ge \frac{1}{192\pi} \bigl(1+\log(1+d)\bigr), \] 
which proves \eqref{eq:F-block}. 
\end{proof}


\begin{proof}[Proof of Theorem~\ref{thm:main}]
The case $n=0$ is immediate. Assume $n\ge1$, and let
\[
[\ell_1,\alpha_1],\ldots,[\ell_s,\alpha_s]
\]
be the maximal blocks of consecutive nonzero digits of $n$. Then it is easy to check that
\begin{equation}\label{eq:vblocks}
v(n)=2s.
\end{equation}

\medskip
\noindent\textit{Lower estimate.}
If $j=\ell_\nu$ is the beginning of a block, then either $j=0$ and $t_j=0$,
or $j\ge1$ and $n_{j-1}=0$. In the latter case $N_{j-1}<M_{j-1}$, and hence
\[
t_j=\frac{N_{j-1}}{M_j}<\frac1{m_{j-1}}\le\frac12.
\]
Using Propositions~\ref{prop:recursion} and~\ref{prop:F-log}, we obtain
\[
L_n
\ge
\frac{1}{192\pi}
\bigg(
s+
\sum_{\substack{j\ge0\\ n_j\ne0}}
\log(1+d_j)
\bigg).
\]
By using \eqref{eq:vblocks}
we have
$
s=\frac{v(n)}{2}.
$
Therefore,
\[
\frac{1}{192\pi}
\bigg(
\frac{v(n)}{2}
+
\sum_{\substack{j\ge0\\ n_j\ne0}}
\log(1+d_j)
\bigg)
\ge
\frac{1}{384\pi}
V_{\log}(n).
\]

\medskip
\noindent\textit{Upper estimate.}
Dirichlet-kernel identity \eqref{eq:blockD1}
allows us to write 
\[
D_n=\psi_n
\sum_{\substack{j\ge0\\ n_j\ne0}}
D_{M_j}
\sum_{u=1}^{m_j-1}r_j^u-\psi_n
\sum_{\substack{j\ge0\\ n_j\ne0}}
D_{M_j}
\sum_{u=1}^{m_j-n_j-1}r_j^u=I-II.
\]

As in the blockwise telescoping argument of \cite{BPT}, we have 
\begin{equation}\label{eq:I-upper} 
\|I\|_1\le v(n). 
\end{equation} 
The telescoping identity used in this argument is purely algebraic and does not require the boundedness condition \(\sup_jm_j<\infty\).

For every \(j\) such that \(n_j\ne0\), put $ \rho_j:=m_j-n_j-1. $ Then 
\[ \left\| D_{M_j}\sum_{u=1}^{\rho_j}r_j^u \right\|_1 = \Lambda_{m_j}(\rho_j). \]
Moreover, 
$$ \min\{\rho_j,m_j-\rho_j\} = \min\{m_j-n_j-1,n_j+1\} = d_j. $$
Hence, by Lemma~\ref{lem:Lambda-upper}, 
\[ \Lambda_{m_j}(\rho_j) \le \frac{7}{2}\log(1+d_j). \] 
Consequently, 
\[ \begin{aligned} 
\|II\|_1 \le \sum_{\substack{j\ge0\\ n_j\ne0}} \Lambda_{m_j}(\rho_j)\le \frac{7}{2} \sum_{\substack{j\ge0\\ n_j\ne0}} \log(1+d_j). 
\end{aligned} \] 
Combining this estimate with \eqref{eq:I-upper}, we obtain 
\[ \begin{aligned} 
L_n =\|D_n\|_1 &\le \|I\|_1+\|II\|_1\le v(n) + \frac{7}{2} \sum_{\substack{j\ge0\\ n_j\ne0}} \log(1+d_j)\\ &\le \frac{7}{2} \bigg( v(n) + \sum_{\substack{j\ge0\\ n_j\ne0}} \log(1+d_j) \bigg)= \frac{7}{2}V_{\log}(n). 
\end{aligned} \] 
Thus, 
$L_n\le\frac{7}{2}V_{\log}(n). $
This completes the proof of the upper estimate.
\end{proof}

\begin{corollary}\label{cor:log-growth} 
For every Vilenkin system and every integer \(n\ge2\), \[ L_n\le C\log(n+1), \] where \(C>0\) is an absolute constant independent of the generating sequence \(\{m_j\}\). Moreover, this logarithmic order is sharp. In particular, for the Walsh system, if $q_k:=2^{2k}+2^{2k-2}+\cdots+1,$ then \[ L_{q_k}\asymp \log q_k . \] 
\end{corollary} 
\begin{proof} Let \(J:=|n|\). Since \(M_J\le n\) and \(M_J\ge2^J\), we have $ J\le\log_2 n. $ Moreover, $ v(n)\le2(J+1). $ For \(j<J\), $ 1+d_j\le m_j,$ while for the highest nonzero digit, 
$ 1+d_J\le n_J+2\le3n_J. $
Consequently, 
$$ \sum_{\substack{j\ge0\\ n_j\ne0}} \log(1+d_j) \le \log M_J+\log(3n_J)\le \log(3n). $$
Hence $V_{\log}(n) \le 2\log_2 n+2+\log(3n) \le C_1\log(n+1)$ with an absolute constant \(C_1\). Therefore, by Theorem~\ref{thm:main}, $ L_n\le C\log(n+1).$ 

To see that this order cannot be improved in general, consider the Walsh system and 
$q_k=\sum_{j=0}^{k}2^{2j} =\frac{4^{k+1}-1}{3}.$ The nonzero binary digits of \(q_k\) are isolated, and therefore $ v(q_k)=2(k+1).$ Since \(m_j=2\), we have \(d_j=0\) at every nonzero digit, and hence $ V_{\log}(q_k)=2(k+1).$ 
By Theorem~\ref{thm:main}, $ L_{q_k}\asymp k+1.$ 
On the other hand, $\log q_k\asymp k+1,$ and consequently $L_{q_k}\asymp\log q_k.$ 
\end{proof}

\begin{remark}
Theorem~\ref{thm:main} does not require $\sup_jm_j<\infty$. Thus the
constants in the two-sided estimate are independent of the local bases
$m_j$. In particular, when some $m_j$ are large, the contribution of the
$j$th digit is logarithmic in the shorter cyclic length $d_j$, rather than
linear in $m_j-n_j-1$.
\end{remark}

\begin{remark}
We have not attempted to optimize the absolute constants in
Theorem~\ref{thm:main}. The point is the uniform equivalence
$L_n\asymp \Vlog(n)$ for arbitrary Vilenkin systems.
\end{remark}

\section{Norm convergence of subsequences in $L_1$}

We first derive a quantitative estimate in terms of the modulus of
continuity introduced in \eqref{def:omega-p}.

\begin{proposition}\label{prop:L1-modulus}
Let $f\in L_1(G_m)$ and let $n\ge1$. If
$
M_r\le n<M_{r+1},
$
then
\begin{equation}\label{eq:L1-modulus}
\|S_nf-f\|_1
\le
4\,\Vlog(n)\,
\omega_1\left(\frac1{M_r},f\right).
\end{equation}
\end{proposition}

\begin{proof}
Since $n\ge M_r$, we have
$$
S_nS_{M_r}f=S_{M_r}f.
$$
Consequently,
\[
S_nf-f
=
S_n\bigl(f-S_{M_r}f\bigr)
+S_{M_r}f-f.
\]
By Young's inequality and the definition of the Lebesgue constant,
\[
\|S_ng\|_1\le L_n\|g\|_1,
\qquad g\in L_1(G_m).
\]
Hence
\[
\|S_nf-f\|_1
\le
(L_n+1)\|f-S_{M_r}f\|_1.
\]
Using \eqref{eq:omega-SM} and Theorem~\ref{thm:main}, we obtain
\[
\|S_nf-f\|_1
\le
\left(\frac72\Vlog(n)+1\right)
\omega_1\left(\frac1{M_r},f\right).
\]
For $n\ge1$, the set of nonzero digits has at least one block, and hence
$v(n)\ge2$. Therefore 
$\Vlog(n)\ge2$, so
$$
\frac72\Vlog(n)+1\le4\Vlog(n).
$$
This proves \eqref{eq:L1-modulus}.
\end{proof}

The preceding estimate yields a characterization of norm convergence along
subsequences.

\begin{theorem}\label{thm:L1-subsequence}
Let $G_m$ be an arbitrary Vilenkin group, bounded or unbounded, and let
$\{\alpha_k:k\in\N\}$ be a strictly increasing sequence of positive
integers. Then the following conditions are equivalent:
\begin{enumerate}[label=\textup{(\roman*)}]
\item \label{item:Vlog-L1}
$\displaystyle \sup_{k\in\N}\Vlog(\alpha_k)<\infty$;
\item \label{item:uniform-L1}
$\displaystyle \sup_{k\in\N}
\|S_{\alpha_k}\|_{L_1\to L_1}<\infty$;
\item \label{item:conv-L1}
for every $f\in L_1(G_m)$,
\[
\|S_{\alpha_k}f-f\|_1\rightarrow0
\quad\text{as } \quad k\to\infty.
\]
\end{enumerate}
In particular, bounded logarithmic variation of the indices is sufficient
for $L_1$-norm convergence of the corresponding subsequence of partial sums.
\end{theorem}

\begin{proof}
Assume first \ref{item:Vlog-L1}, and put
$$
K:=\sup_k\Vlog(\alpha_k)<\infty.
$$
Let $r_k:=|\alpha_k|$, so that
$
M_{r_k}\le \alpha_k<M_{r_k+1}.
$
Since $\alpha_k\to\infty$, we have $r_k\to\infty$. By
Proposition~\ref{prop:L1-modulus},
\[
\|S_{\alpha_k}f-f\|_1
\le
4K\,
\omega_1\left(\frac1{M_{r_k}},f\right).
\]
The modulus of continuity tends to zero as $r_k\to\infty$, and hence
\[
\|S_{\alpha_k}f-f\|_1\to0.
\]
Thus \ref{item:Vlog-L1} implies \ref{item:conv-L1}.

The implication \ref{item:conv-L1}$\Rightarrow$\ref{item:uniform-L1}
follows from the uniform boundedness principle, since for every
$f\in L_1(G_m)$ the sequence $\{S_{\alpha_k}f\}$ is norm convergent and is
therefore bounded in $L_1$.

It remains to prove
\ref{item:uniform-L1}$\Rightarrow$\ref{item:Vlog-L1}. We first observe that
for every $n$,
\begin{equation}\label{eq:operator-norm-Ln}
\|S_n\|_{L_1\to L_1}=L_n.
\end{equation}
Indeed, since $S_nf=f*D_n$, Young's inequality gives
\[
\|S_n\|_{L_1\to L_1}\le\|D_n\|_1=L_n.
\]
Conversely, choose $r$ so large that $M_r>n$. 
Since the first $n$ Fourier coefficients of $D_{M_r}$ are equal to one, we
have
$
S_nD_{M_r}=D_n.
$
Thus, by combining \eqref{eq:DM} and \eqref{eq:normDM} we get that
\[
\|S_n\|_{L_1\to L_1}
\ge
\frac{\|S_nD_{M_r}\|_1}{\|D_{M_r}\|_1}
=L_n,
\]
which proves \eqref{eq:operator-norm-Ln}.

If \ref{item:uniform-L1} holds, then by \eqref{eq:operator-norm-Ln} the
sequence $\{L_{\alpha_k}\}$ is bounded. The lower estimate in
Theorem~\ref{thm:main} gives
\[
\frac1{384\pi}\Vlog(\alpha_k)
\le L_{\alpha_k} \ \ \ \text{
and }\ \ \
\sup_k\Vlog(\alpha_k)<\infty.
\]
This proves \ref{item:Vlog-L1} and completes the proof.
\end{proof}

\begin{remark}
For bounded Vilenkin systems, norm-convergence criteria of this type were
previously expressed through the classical digit variations; compare
\cite{BPT}. Theorem~\ref{thm:L1-subsequence} shows that the logarithmic
variation is the natural quantity for arbitrary Vilenkin systems, since its
constants are independent of $\sup_jm_j$.
\end{remark}

\section{Almost everywhere convergence of subsequences}

We now prove Theorem~\ref{thm:ae}. The argument follows the maximal-operator
strategy used in \cite{APT}, but the decomposition is adapted to the
logarithmic variation $\Vlog$.
For $n=\sum_jn_jM_j$ and $n_j\ne0$, recall that
$
d_j=\min\{m_j-n_j-1,n_j+1\}.
$
Split the set of nonzero digits into
\[
\mathcal A(n):=
\{j:n_j\ne0,\ n_j+1\le m_j-n_j-1\}
\]
and
\[
\mathcal B(n):=
\{j:n_j\ne0,\ m_j-n_j-1<n_j+1\}.
\]
Thus, for $j\in\mathcal A(n)$, one has $d_j=n_j+1$, whereas for
$j\in\mathcal B(n)$,
\[
d_j=m_j-n_j-1.
\]

\begin{proposition}\label{prop:pointwise-domination}
For every $K>0$ there exists a constant $A_K>0$ with the following
property. If $n\ge1$ and $\Vlog(n)\le K$, then for every
$f\in L^1(G_m)$,
\begin{equation}\label{eq:pointwise-domination}
|S_nf(x)|\le A_K\,\mathcal Mf(x),
\qquad x\in G_m.
\end{equation}
The constant $A_K$ is independent of the generating sequence $m$.
\end{proposition}

\begin{proof}
Starting from \eqref{eq:blockD1}, for $j\in\mathcal A(n)$ we keep the
short sum in its original form, while for $j\in\mathcal B(n)$ we use
\[
\sum_{u=m_j-n_j}^{m_j-1}r_j^u
=
\sum_{u=1}^{m_j-1}r_j^u
-
\sum_{u=1}^{m_j-n_j-1}r_j^u.
\]
Hence
\begin{equation}\label{eq:Dn-ABC}
D_n=\psi_n\bigl(\mathcal P_n+\mathcal Q_n-\mathcal R_n\bigr),
\end{equation}
where
\[
\mathcal P_n:=
\sum_{j\in\mathcal A(n)}D_{M_j}
\sum_{u=m_j-n_j}^{m_j-1}r_j^u,
\]
\[
\mathcal Q_n:=
\sum_{j\in\mathcal B(n)}D_{M_j}
\sum_{u=1}^{m_j-1}r_j^u,
\]
and
\[
\mathcal R_n:=
\sum_{j\in\mathcal B(n)}D_{M_j}
\sum_{u=1}^{m_j-n_j-1}r_j^u.
\]

Write $\mathcal B(n)$ as the union of its maximal consecutive intervals
\[
[a_1,b_1],\ldots,[a_{R_\mathcal B},b_{R_\mathcal B}].
\]
For each such interval, the same telescoping identity used in the proof of
the upper estimate in \cite{BPT,APT} gives
\begin{equation}\label{eq:full-block-telescope}
\sum_{j=a_\nu}^{b_\nu}D_{M_j}
\sum_{u=1}^{m_j-1}r_j^u
=
\frac{D_{M_{b_\nu+1}}}{\psi_{M_{b_\nu+1}-1}}
-
\frac{D_{M_{a_\nu}}}{\psi_{M_{a_\nu}-1}}.
\end{equation}
Consequently, after convolution with $f$ and taking absolute values, every
interval $[a_\nu,b_\nu]$ contributes at most two copies of
$\mathcal Mf(x)$.

The remaining terms in \eqref{eq:Dn-ABC} are short sums. Their total
number of individual Rademacher terms is at most
\begin{equation}\label{eq:short-count}
Q(n):=
\sum_{j\in\mathcal A(n)}n_j
+
\sum_{j\in\mathcal B(n)}(m_j-n_j-1)
\le
\sum_{\substack{j\ge0\\ n_j\ne0}}d_j.
\end{equation}
Indeed, if $j\in\mathcal A(n)$, then $n_j<d_j=n_j+1$, while if
$j\in\mathcal B(n)$ the second summand is exactly $d_j$.
Each individual term has the form $D_{M_j}$ multiplied by a character of
modulus one, and therefore its convolution with $f$ is bounded in absolute
value by $\mathcal Mf(x)$. It remains to bound uniformly the number of terms in this decomposition.

Let
$
P(n):=\#\{j:n_j\ne0,\ d_j\ge1\}.
$
Since 
$\log(1+d_j)\ge\log2 $ whenever $d_j\ge1,$
the assumption
$\Vlog(n)\le K$ implies
\[
\sum_{\substack{j\ge0\\ n_j\ne0}}
\log(1+d_j)
\ge
P(n)\log 2,
\]
and
$$
P(n)\le\frac{K}{\log2}.
$$
Moreover, for every such $j$,
\[
\log(1+d_j)
\le
\sum_{\substack{i\ge0\\ n_i\ne0}}
\log(1+d_i)
\le
V_{\log}(n)
\le
K,  \ \ \ d_j\le e^K-1,
\]
and hence
\begin{equation*}
\sum_{\substack{j\ge0\\ n_j\ne0}}d_j
\le (e^K-1)P(n)\le
(e^K-1)\frac{K}{\log2}.
\end{equation*}
If $s$ denotes the number of maximal blocks of nonzero digits of $n$, then
\eqref{eq:vblocks} and $\Vlog(n)\le K$ give
$$
s=\frac{v(n)}2\le\frac K2.
$$
Inside each nonzero digit block, the maximal intervals of
$\mathcal B(n)$ can be separated only by indices belonging to
$\mathcal A(n)$.  Thus, if the \(s\) nonzero-digit blocks contain respectively
\(q_1,\ldots,q_s\) indices from \(\mathcal A(n)\), then
\[
R_{\mathcal B}
\le
\sum_{\nu=1}^s (q_\nu+1)
=
s+\sum_{\nu=1}^s q_\nu
=
s+\#\mathcal A(n).
\]
Therefore,
\begin{equation}\label{eq:B-runs}
R_{\mathcal B}\le s+\#\mathcal A(n)
\le \frac K2+\frac{K}{\log2}.
\end{equation}
Combining \eqref{eq:short-count}--\eqref{eq:B-runs}, we obtain
\[
2R_{\mathcal B}+Q(n)
\le
K+\frac{2K}{\log2}
+(e^K-1)\frac{K}{\log2}
=:A_K.
\]
Since
$$
S_nf(x)=\int_{G_m}f(x-t)D_n(t)\,d\mu(t),
$$ 
by using
\eqref{eq:Dn-ABC} and
\eqref{eq:full-block-telescope}, we obtain
\[
|S_nf(x)|\le (2R_{\mathcal B}+Q(n))\mathcal Mf(x)
\le A_K\mathcal Mf(x),
\]
which proves \eqref{eq:pointwise-domination}. The proof is complete.
\end{proof}

\begin{proof}[Proof of Theorem~\ref{thm:ae}]
By Proposition~\ref{prop:pointwise-domination} and
\eqref{eq:Vlog-bounded}, we can conclude that
$$
\sup_{k\in\N}|S_{\alpha_k}f(x)|
\le A_K\mathcal Mf(x).
$$
Therefore \eqref{eq:mart-max-weak} gives
\[
\sup_{\lambda>0}\lambda\,
\mu\left\{x:\sup_{k\in\N}|S_{\alpha_k}f(x)|>\lambda\right\}
\le C A_K\|f\|_1,
\]
which proves \eqref{eq:weak-subsequence}.

To prove almost everywhere convergence, let $P$ be a Vilenkin polynomial.
Since $\alpha_k\to\infty$, one has
$
S_{\alpha_k}P=P
$
for all sufficiently large $k$. Vilenkin polynomials are dense in
$L^1(G_m)$. 

Fix $f\in L^1(G_m)$ and $\varepsilon>0$. For every Vilenkin
polynomial $P$,
\[
\limsup_{k\to\infty}|S_{\alpha_k}f-f|
\le
\sup_k|S_{\alpha_k}(f-P)|+|f-P|, \ \ \ \ \text{a.e.}
\]
Hence, by \eqref{eq:weak-subsequence} and Chebyshev's inequality,
\[
\begin{aligned}
&\mu\left\{x:
\limsup_{k\to\infty}|S_{\alpha_k}f(x)-f(x)|>2\varepsilon\right\}\\
&\le
\mu\left\{x:\sup_k|S_{\alpha_k}(f-P)(x)|>\varepsilon\right\}
+
\mu\{|f-P|>\varepsilon\}\\
&\le
\frac{C_K+1}{\varepsilon}\|f-P\|_1.
\end{aligned}
\]
Letting $P\to f$ in $L^1(G_m)$ shows that the measure on the left is zero.
Since $\varepsilon>0$ is arbitrary, \eqref{eq:ae-subsequence} follows.
\end{proof}

\begin{remark}
Theorem~\ref{thm:ae} recovers, in particular, the classical almost
everywhere convergence of the martingale subsequence $S_{M_k}f$ for every
$f\in L^1(G_m)$. In the unbounded case it also shows why a logarithmic digit
criterion is more natural than a condition involving the linear quantity
$v^*$: the digits $n_j=1$ may have very large $v^*$-contribution when
$m_j$ is large, while the corresponding Dirichlet-kernel piece consists of
only one martingale kernel.
\end{remark}

\end{document}